\RequirePackage{fix-cm}
\documentclass[smallcondensed]{svjour3}     % onecolumn (ditto)
\smartqed  % flush right qed marks, e.g. at end of proof
\usepackage{graphicx}
\usepackage{framed}
\usepackage{latexsym,amsfonts,amsbsy,amssymb}
\usepackage{amsmath}
\usepackage[colorlinks, citecolor=blue]{hyperref}
\usepackage{cite}

\newtheorem{algorithm}{Algorithm}[section]

\def\cA{\mathcal{A}}

\def\cF{\mathcal{F}}
\def\cG{\mathcal{G}}

\def\cO{\mathcal{O}}

\def\R{\mathbb{R}}

\def\w\eta{\widetilde{\eta}}

\journalname{Journal of Scientific Computing}
\begin{document}

\title{Multigrid method for nonlinear eigenvalue problems based on Newton iteration\thanks{This work was supported by General projects
of science and technology plan of Beijing Municipal Education Commission (Grant No. KM202110005011), National Natural Science Foundation of China (Grant Nos. 11801021).}
}

\titlerunning{Multigrid method for nonlinear eigenvalue problems based on Newton iteration}        % if too long for running head

\author{Fei Xu        \and
        Manting Xie
        \and
        Meiling Yue
}


\institute{   F. Xu \at
              Institute of Computational Mathematics, Department of Mathematics, Faculty of Science, Beijing University of Technology, Beijing, 100124, China. \\
              \email{xufei@lsec.cc.ac.cn}           %  \\
              \and
              M. Xie \at
              Center for Applied Mathematics, Tianjin University, Tianjin 300072, China. \\
              \email{mtxie@tju.edu.cn}           %  \\         %
              \and
              M. Yue \at
              School of Mathematics and Statistics, Beijing Technology and Business University, Beijing 100048, China.
              \email{yuemeiling@lsec.cc.ac.cn}
}

\date{Received: date / Accepted: date}

\maketitle

\begin{abstract}
In this paper, a novel multigrid method based on Newton iteration is proposed to solve nonlinear eigenvalue problems.
Instead of handling the eigenvalue $\lambda$ and eigenfunction $u$ separately, we treat the eigenpair $(\lambda, u)$
as one element in a product space $\mathbb R \times H_0^1(\Omega)$. Then in the presented multigrid method, only one discrete
linear boundary value problem needs to be solved for each level of the multigrid sequence.
Because we avoid solving large-scale nonlinear eigenvalue problems directly, the overall efficiency is significantly improved.
The optimal error estimate and linear computational complexity can be derived simultaneously.
In addition, we also provide an improved multigrid method coupled with a mixing scheme to further guarantee
the convergence and stability of the iteration scheme.
More importantly, we prove convergence for the residuals after each iteration step.
For nonlinear eigenvalue problems, such theoretical analysis
is missing from the existing literatures on the mixing iteration scheme.
\end{abstract}

%====================================================================================================================
\section{Introduction}
%Solving large-scale nonlinear eigenvalue problems is a fundamental and challenging task in modern science and engineering computing.
Solving large-scale nonlinear eigenvalue problems is a basic and challenging task in the field of scientific and engineering computing.
Various practical problems reduce to solve nonlinear eigenvalue problems.
For example, the Schr\"{o}dinger-Newton equation which has been widely used to model the quantum state reduction \cite{harri},
the Gross-Pitaevskii equation which describes Bose-Einstein condensates \cite{baowei}, and the Kohn-Sham equation
which is used to describe the electronic structure in quantum chemistry \cite{CancesChakirMaday,Canceschakir}.
In this paper, we will study a class of nonlinear eigenvalue models with convex energy functional.
It is quite difficult to solve such nonlinear eigenvalue problems whose computational work grows exponentially as the problem size increases.
Additionally, there are far fewer available numerical methods for nonlinear eigenvalue problems than those for the boundary value problems.

In this study, we resort to the multigrid method to improve the solving efficiency for nonlinear eigenvalue problems.
The multigrid method was first proposed by Fedorenko \cite{fedore} in 1961. Later, Brandt \cite{brandt} demonstrated the efficiency of the multigrid method,
which was then a widely examined subject in the computational mathematics field.
Subsequently, Hackbusch, Xu, and many others began to use the tool of functional analysis to analyze this algorithm \cite{Hackbusch_Book,Xu,Shaidurov},
which made the multigrid method experience rapid development.
As we know, the multigrid method is able to derive an approximate solution possessing the optimal error estimates with
the linear computational complexity.
For eigenvalue problems, the study on the multigrid method is relatively limited.
Xu and Zhou \cite{XuZhou_2001} first designed a two-grid method for the linear eigenvalue problem. This algorithm requires solving a small-scale eigenvalue problem on
a coarse mesh and a large-scale boundary value problem on a fine mesh. When the mesh sizes of the coarse mesh ($H$) and fine mesh ($h$) have an appropriate proportional relation ($H=\sqrt{h}$),
the optimal estimate for the approximate solution can be derived.
Later, Chen and Liu et al. \cite{chenliu} extended the two-grid method to solve nonlinear eigenvalue problems using specially selected coarse and fine meshes.
However, owing to the strict constrains on the ratio (i.e., $H=\sqrt{h}$), the two-grid method only performs on two mesh levels and cannot be used in the multigrid sequence.

To solve the nonlinear eigenvalue problems via the multigrid technique, the Newton iteration is adopted.
In this paper, the eigenvalue $\lambda$ and
eigenfunction $u$ are treated as one element in a product space $\mathbb R\times H_0^1(\Omega)$.
Then, the nonlinear eigenvalue equation can be viewed as a special nonlinear equation defined
in $\mathbb R\times H_0^1(\Omega)$. Next, based on the Newton iteration technique, we simply solve a linear
boundary value problem in each layer of the multigrid sequence. Solving the large-scale nonlinear eigenvalue problem is avoided in our algorithm. Besides, the involved
linearized boundary value problems can be solved quickly by many mature numerical algorithms, such as the classical multigrid method.
Hence, the presented algorithm can significantly {\color{blue}improve} the solving efficiency for nonlinear eigenvalue problems.
The well-posedness for these linearized boundary value problems and the convergence for our entire algorithm {\color{blue}have} been rigorously
analyzed in this paper.

For the above-mentioned multigrid method, we find that the Newton iteration scheme may fail to converge for some complicated nonlinear eigenvalue models.
In fact, this problem exists widely during the self-consistent field iteration for nonlinear eigenvalue problems.
To derive a convergence result, the mixing scheme (or damping method) is introduced.
The most widely used mixing scheme is Anderson acceleration, which is used to improve the convergence rate for fixed-point iterations.
Anderson acceleration was first designed by D. G. Anderson in 1965 to solve the integral equation \cite{andersion}.
Subsequently, this technique was used to solve various models, such as, the self-consistent field iteration in electronic structure computations \cite{fang},
flow problems \cite{vc1,vc2}, molecular interaction \cite{stasiak}, etc.
Anderson acceleration adopts the mixing scheme to make the fixed-point iterations converge for general equations.
Despite its widespread applications, the first mathematical convergence
result for linear and nonlinear problems did not appear until 2015 in \cite{Toth}.
For nonlinear eigenvalue problems, there is still no strictly theoretical analysis.
In general, the mixing scheme is used in fixed-point iterations to accelerate the convergence rate.
Thus, we use the mixing scheme in the Newton iteration scheme to design a more efficient multigrid method.
By treating the nonlinear eigenvalue problem as a special nonlinear equation, we just solve a series of linearized boundary value problems derived from the
Newton iteration scheme defined
on a multigrid space sequence. Next, we mix the approximate solution in each iteration step to generate a more accurate approximation.
More importantly, we can rigorously {\color{blue}prove convergence} for the residuals of the nonlinear eigenvalue problem after a one-time iteration step.
This may provide some inspiration to prove the theoretical conclusions of the self-consistent field iteration for nonlinear eigenvalue problems.

The remainder of this paper is organized as follows. In Section 2, some preliminaries about the nonlinear eigenvalue problems and the corresponding finite element method are presented.
In Section 3, we use the Newton iteration to solve the nonlinear eigenvalue problems and provide several useful error estimates.
In Section 4, the multigrid method based on Newton iteration is introduced.
In Section 5, we adopt the mixing scheme to provide an improved multigrid algorithm. Numerical experiments are presented in Section 6 to verify the
theoretical results derived in this paper. Some concluding remarks are presented in the last section.

\section{Preliminaries of nonlinear eigenvalue problems}
In this section, we first introduce some standard notations
for Sobolev space $W^{s,p}(\Omega)$
and the associated norm $\|\cdot\|_{s,p,\Omega}$ on a bounded domain $\Omega \in \mathbb R^d \ (d\geq 1)$.
For $p = 2$, we denote $\|\cdot\|_{s,\Omega}=\|\cdot\|_{s,2,\Omega}$,
$H^s(\Omega) = W^{s,2}(\Omega)$ and $H_0^1(\Omega) = \{v \in H^1(\Omega):v|_{\partial \Omega} = 0\}$.
Hereafter, we use $\lesssim$, $\gtrsim$, $\approx$ to denote $\leq$, $\geq$, $=$ with a mesh-independent coefficient for simplicity.

In this paper, we focus on the nonlinear eigenvalue problems arising from the following variational problem:
\begin{eqnarray}\label{min}
I=\inf \{ E(v),v \in H_0^1(\Omega), \int_{\Omega} v^2d\Omega=1\},
\end{eqnarray}
where the energy functional $E$ is defined by
\begin{eqnarray*}
E(v)=\dfrac{1}{2}a(v,v)+\dfrac{1}{2}\int_{\Omega}F(v^2)d\Omega,
\end{eqnarray*}
with
\begin{eqnarray*}
a(u,v) = \int_{\Omega}(\cA\nabla u \cdot \nabla v +Vuv)d\Omega.
\end{eqnarray*}

To simplify the notation, let us define $f(t)=F'(t)$. Then, for all $v \in H_0^1(\Omega)$, $E'(v)=A_vv$ holds with
\begin{eqnarray}
A_v=-{\rm div} (\mathcal A\nabla \cdot)+V+f(v^2).
\end{eqnarray}

Next, applying the Lagrange multiplier method to (\ref{min}) yields the following
nonlinear eigenvalue problem:
%Find $(\lambda,u) \in \mathbb R \times H_0^1(\Omega)$ such that
\begin{equation}\label{Nonlinear_Eigenvalue_Problem}
\left\{
\begin{array}{rcl}
A_uu &=&\lambda u,  \ \ \text{ in } \Omega,\\
u &=&0,  \ \ \ \ \text{ on } \partial\Omega,\\
\int_\Omega u^2 d\Omega &=&1,
\end{array}
\right.
\end{equation}
where $\lambda$ comes from the Lagrange multiplier of the constraint $\|u\|_0^2=1$.

For future theoretical analysis, we assume that the following conditions hold true (see \cite{CancesChakirMaday,ChenZhou}):
\begin{itemize}
\item
$\mathcal A$ is a symmetric positive-definite matrix;
\item
$V \in L^p(\Omega)$ for some $p>\max(1,d/2)$;
\item
$F \in C^1([0,\infty),\mathbb R)\cap C^2((0,\infty),\mathbb R)$ and $F''>0$ on $(0,\infty)$;
\item
$\exists \ 0\leq m <2 \ \text{and}\ C \in \mathbb R_+, \ \text{s.t.}  \ \forall t \geq 0, \ | F'(t)| \leq C(1+t^m)$;
\item
$F''(t)t$ is locally bounded in $[0,\infty)$;
\item
$\exists \ 1<r\leq 2$ and $0\leq s \leq 5-r$ such that $\forall  a >0 , \exists \ C \in \mathbb R_+, \
\text{s.t.} \ \forall \ 0<t_1 \leq a, \ \forall \ t_2 \in \mathbb R,$
 \begin{eqnarray*}
 | f(t_2^2)t_2-f(t_1^2)t_1-2f'(t_1^2)t_1^2(t_2-t_1)| \leq C(1+| t_2|^s)| t_2-t_1|^r;
 \end{eqnarray*}
 \item $\exists \ 0\leq q <1\ \text{and}\  C \in \mathbb R_+, \ \text{s.t.} \ \forall t \geq 0, \ |f'(t)|+|f''(t)t|\leq C(1+t^q)$.
\end{itemize}

In the remainder of this section, we present two lemmas for the nonlinear eigenvalue problem (\ref{Nonlinear_Eigenvalue_Problem}),
and the detailed proof can be found in \cite{CancesChakirMaday}.
\begin{lemma}
There exist $M \in \mathbb R_+$ and $\beta>0$ such that for all $v \in H_0^1(\Omega)$, there holds
\begin{eqnarray}
0\leq ( (A_u-\lambda)v,v)\leq M\|v\|_1^2,
\end{eqnarray}
and
\begin{eqnarray}\label{23}
\beta\|v\|_1^2 \leq ( (E''(u)-\lambda)v,v)\leq M\|v\|_1^2,
\end{eqnarray}
where $E''(u)$ has the following form
\begin{eqnarray}
( E''(u)w,v)=( A_uw,v)+2(f'(u^2)u^2w,v).
\end{eqnarray}
\end{lemma}

Now, we introduce the finite element method for the nonlinear eigenvalue problem (\ref{Nonlinear_Eigenvalue_Problem}). To use the finite element method,
we define the variational form for (\ref{Nonlinear_Eigenvalue_Problem}) as follows: Find $(\lambda,u)\in\mathbb R\times H_0^1(\Omega)$ such that $\|u\|_0=1$ and
\begin{eqnarray}\label{Nonlinear_Eigenvalue_Problemweak}
(\cA\nabla u,\nabla v)+(Vu+f(u^2)u,v)=\lambda (u,v),\quad \forall v\in H_0^1(\Omega).
\end{eqnarray}

Next, we define the finite element space $V_h$ on a triangulation $\mathcal T_h$.
The finite element space $V_h$ is composed of piecewise polynomials such that $V_h\subset H_0^1(\Omega)$ and
\begin{equation}
\lim_{h\rightarrow 0}\inf_{v_h \in V_h} \|w - v_h\|_1 = 0,\ \ \ \forall w\in H_0^1(\Omega).
\end{equation}

Based on the finite element space $V_h$, we can derive the finite element solution for (\ref{Nonlinear_Eigenvalue_Problemweak}) by solving the discrete nonlinear eigenvalue problem as follows:
Find $(\lambda_h,u_h)\in\mathbb R\times V_h$ such that $\|u_h\|_0=1$ and
\begin{equation}\label{GPEfem}
(\cA\nabla u_h,\nabla v_h)+(Vu_h+f(u_h^2)u_h,v_h)=\lambda_h (u_h,v_h),\quad \forall v_h\in V_h.
\end{equation}

The standard error estimates for the finite element approximate eigenpair $(\lambda_h,u_h)$ are described in the following lemma.
\begin{lemma}\label{stalem}
There exists $h_0>0$, such that for all $0<h<h_0$, the following error estimates hold true
%\begin{eqnarray}
%\dfrac{\gamma}{2}\|u_h-u\|_1^2\leq E(u_h)-E(u)\leq\dfrac{M}{2}\|u_h-u\|_1^2 +C\|u_h-u\|_{L^{6/(5-2q)}}
%\end{eqnarray}
\begin{eqnarray}
\|u-u_h\|_1 \lesssim \delta_h(u),
\end{eqnarray}
and
\begin{eqnarray}
| \lambda-\lambda_h| \lesssim \|u-u_h\|_1^2+\|u-u_h\|_{L^{6/(5-2m)}},
\end{eqnarray}
where
\begin{equation}\label{delta}
\delta_h(u) := \inf_{v_h\in V_h}\|u - v_h\|_1.
\end{equation}
%\begin{eqnarray}
%\|u_h-u\|_0^2 \leq C(\|u_h-u\|_0\|u_h-u\|_{L^{6r/(5-s)}}^r+\|u_h-u\|_1\mathop{\min}\limits_{\psi_h \in X_h}\|\psi_{u_h-u}-\psi_h\|_1).
%\end{eqnarray}
%If $F''$ is bounded, there exists $C \in \cR_+$ such that for all $h>0$,
%\begin{eqnarray}
%\dfrac{\gamma}{2}\|u_h-u\|_1^2\leq E(u_h)-E(u)\leq C\|u_h-u\|_1^2.
%\end{eqnarray}
%Lastly, if $0<r+s\leq 3$, then
%\begin{eqnarray}
%\|u_h-u\|_{L^{6r/(5-s)}}^r\leq \|u_h-u\|_0\|u_h-u\|_1^{r-1}.
%\end{eqnarray}
%so this implies
%\begin{eqnarray}
%\|u_h-u\|_0\leq C\|u_h-u\|_1\mathop{\min}\limits_{\psi_h \in X_h}\|\psi_{u_h-u}-\psi_h\|_1.
%\end{eqnarray}
\end{lemma}

\section{Newton iteration for nonlinear eigenvalue problems}
To introduce Newton iteration for nonlinear eigenvalue problem (\ref{Nonlinear_Eigenvalue_Problemweak}), we first
give some symbols to simplify the description.
Let us define
\begin{eqnarray}
\langle\cF(\lambda,u),v\rangle = (\cA\nabla u,\nabla v)+(Vu+f(u^2)u-\lambda u,v), \quad \forall v\in H_0^1(\Omega).
\end{eqnarray}
The Fr\'{e}chet derivation of $\cF$ with respect to $u$ at $(\lambda,u)$ is denoted by
\begin{eqnarray}
\langle\cF'_u(\lambda,u)w,v\rangle &=& (\cA\nabla w,\nabla v)+((V+f(u^2)-\lambda)w,v)+2(f'(u^2)u^2w,v)\nonumber \\
&=&(E''(u)w,v)-\lambda(w,v).
\end{eqnarray}

Next, we denote the product space $\mathbb R\times H_0^1(\Omega)$ by $X$. Let us define the norms in $X$ as
\begin{eqnarray}
\|(\lambda,u)\|_X=|\lambda |+\|u\|_1 \quad
\text{and} \quad
\|(\lambda,u)\|_0=|\lambda |+\|u\|_0, \quad \forall (\lambda,u)\in X.
\end{eqnarray}
Let us define
\begin{eqnarray}\label{definition of g}
 \langle\cG(\lambda,u),(\mu,v)\rangle=\langle \cF(\lambda,u),v\rangle+\frac{1}{2}\mu(1-\int_{\Omega}u^2 d\Omega),\quad \forall (\mu ,v) \in X.
\end{eqnarray}
Obviously, the nonlinear eigenvalue problem (\ref{Nonlinear_Eigenvalue_Problemweak}) is equivalent to the following nonlinear equation defined
in the product space $X$:
Find $(\lambda,u)\in X$ such that
\begin{eqnarray}\label{exanon}
 \langle\cG(\lambda,u),(\mu,v)\rangle=0,\quad \forall (\mu ,v) \in X.
\end{eqnarray}

Similarly, to solve (\ref{exanon}) using the finite element method, we also introduce a finite dimensional space $X_h:=\mathbb R\times V_h$ as the approximation for $X$.
Then based on $X_h$, the finite element equation (\ref{GPEfem}) is equivalent to the following nonlinear equation defined
in the tensor product space $X_h$: Find $(\lambda_h,u_h)\in X_h$ such that
\begin{eqnarray}\label{nonfem}
 \langle\cG(\lambda_h,u_h),(\mu,v_h)\rangle=0,\quad \forall (\mu ,v_h) \in X_h.
\end{eqnarray}

The Fr\'{e}chet derivation of $\cG$ at $(\lambda,u)$ is denoted by:
\begin{eqnarray}\label{definition of g'}
\langle\cG'(\lambda,u)(\gamma,w),(\mu,v)\rangle=\langle \cF_u'(\lambda,u)w,v\rangle-\gamma(u,v)-\mu(u,w),\quad \forall (\mu ,v) \in X.
\end{eqnarray}

Next, we begin to introduce the Newton iteration for (\ref{nonfem}).
Given an initial value $(\lambda^{(0)},u^{(0)})\in X_h$,
the one step of the Newton iteration is defined as follows, which will generate a more accurate approximate solution:
Find $(\lambda^{(1)},u^{(1)})\in X_h$ such that
\begin{eqnarray}\label{Newton Iteration}
\big\langle \cG'(\lambda^{(0)},u^{(0)})(\lambda^{(1)}-\lambda^{(0)},u^{(1)}-u^{(0)}),(\mu,v_h)\big\rangle
=-\big\langle\cG(\lambda^{(0)},u^{(0)}),(\mu,v_h)\big\rangle,  \forall (\mu,v_h)\in X_h.
\end{eqnarray}
Based on the definitions (\ref{definition of g}) and (\ref{definition of g'}), the linearized boundary value problem (\ref{Newton Iteration})
can be rewritten as follows: Find $(\lambda^{(1)},u^{(1)})\in \mathbb R\times V_h$ such that
for any $(\mu,v_h)\in \mathbb R\times V_h$, there holds
\begin{eqnarray}\label{newton step}
\left\{
\begin{array}{lll}
a(\lambda^{(0)}, u^{(0)};u^{(1)},v_h)+b(u^{(0)};v_h,\lambda^{(1)})
&=&\langle \cF'_{u}(\lambda^{(0)}, u^{(0)})u^{(0)},v_h\rangle \\
&&-\langle \cF(\lambda^{(0)}, u^{(0)}),v_h\rangle-\lambda^{(0)}(u^{(0)},v_h),\\
b(u^{(0)};u^{(1)},\mu)&=&-\mu/2-\mu(u^{(0)},u^{(0)})/2,
\end{array}
\right.
\end{eqnarray}
where
\begin{equation*}
a(\lambda^{(0)}, u^{(0)};u^{(1)},v_h)=\langle \cF'_{u}(\lambda^{(0)}, u^{(0)})u^{(1)},v_h\rangle,\ \ \ \  b(u^{(0)};u^{(1)},\mu) = -\mu(u^{(0)},u^{(1)}).
\end{equation*}

%The isomorphism property of $\cG'$ is analyzed in the following theorem which also guarantees the well-posedness of the above mixed problem.

To guarantee the well-posedness of (\ref{Newton Iteration}) and (\ref{newton step}), an isomorphism property is proved in Theorem \ref{coe}:
\begin{theorem}\label{coe}
Let $(\lambda,u)$ be the exact solution of (\ref{exanon}).
When the mesh size $h$ is sufficiently small, we can derive the following isomorphism properties for the operator $\mathcal G'$:
\begin{eqnarray}\label{coercive of g1}
\mathop {\sup} \limits_{(\mu,v)\in X}\dfrac{\langle\mathcal G'(\lambda,u)(\gamma,w),(\mu,v)\rangle}{\|(\mu,v)\|_X}
&\gtrsim&\|(\gamma,w)\|_X, \quad \forall (\gamma,w)\in X,
\end{eqnarray}
and
\begin{eqnarray}\label{coercive of g2}
\mathop {\sup} \limits_{(\mu,v)\in X_h}\dfrac{\langle\mathcal G'(\lambda,u)(\gamma,w),(\mu,v)\rangle}{\|(\mu,v)\|_X}
&\gtrsim&\|(\gamma,w)\|_X, \quad \forall (\gamma,w)\in X_h.
\end{eqnarray}
For any $(\widehat{\lambda},\widehat{u})\in X$ such that
$\|(\widehat{\lambda}-\lambda,\widehat{u}-u)\|_X$ is sufficiently small, we can also derive
\begin{eqnarray}\label{coercive of g3}
\mathop {\sup} \limits_{(\mu,v)\in X_h}\dfrac{\langle\mathcal G'(\widehat{\lambda},\widehat{u})(\gamma,w),(\mu,v)\rangle}{\|(\mu,v)\|_X}
&\gtrsim& \|(\gamma,w)\|_X, \quad \forall (\gamma,w)\in X_h.
\end{eqnarray}
\end{theorem}
\begin{proof}
Firstly, proving (\ref{coercive of g1}) is equivalent to prove that for any $(\tau,f)\in \mathbb R\times H^{-1}(\Omega)$, the following equation
%\begin{eqnarray}
%\cG'(\lambda,u)(\gamma,w)=(\tau,f)
%\end{eqnarray}
is uniquely solvable: Find $(\gamma,w)\in \mathbb R\times H_0^1(\Omega)$ such that
%From (\ref{definition of g'}), we obtain that (\ref{linear equation}) can be rewritten as
\begin{eqnarray}\label{linear equation}
\left\{
\begin{array}{ll}
a(\lambda, u;w,v)+ b(u;v,\gamma)=(f,v),  &\forall v\in H_0^1(\Omega),\\
b(u;w,\mu)=\mu\tau, &\forall \mu \in \mathbb R.
\end{array}
\right.
\end{eqnarray}
According to the Brezzi theory \cite{Fortin}, we just need to prove the following two conclusions:\\
(1) The following equation
\begin{eqnarray}\label{unique of a}
a(\lambda, u; w,v)=(f,v), \quad \forall v\in V_0,
\end{eqnarray}
is uniquely solvable in $V_0:=\{v\in H_0^1(\Omega): b(u;v,\mu)=0,\ \forall \mu \in \R\}$.\\
(2) The following inf-sup condition holds true
\begin{eqnarray}\label{infsup}
\inf_{\mu\in \R}\sup_{v\in H_0^1(\Omega)}\dfrac{b(u;v,\mu)}{\|v\|_1|\mu|}\geq k_b,
\end{eqnarray}
for some constant $k_b>0$.

Firstly, from (\ref{23}),  we can easily find that (\ref{unique of a}) is uniquely solvable.
Next, since $\|u\|_0=1$, taking $v=-\mu u$ leads to
\begin{eqnarray*}
\inf_{\mu\in \R}\sup_{v\in H_0^1(\Omega)}\dfrac{ b(u;v,\mu)}{\|v\|_1|\mu|} =
\frac{1}{\|u\|_1}=: k_b.
\end{eqnarray*}
Then we derive the inf-sup condition (\ref{infsup}).
Thus, we can derive the well-posedness of (\ref{linear equation}) based on Brezzi theory, which further yields (\ref{coercive of g1}).

Next, we begin to prove the second formula (\ref{coercive of g2}).
Based on (\ref{23}), let us define the projection $P_h: H_0^1(\Omega)\rightarrow V_h$ by
\begin{eqnarray}\label{ph}
a(\lambda, u;w,v-P_hv)=0, \quad \forall w\in V_h,\ \forall v\in H_0^1(\Omega).
\end{eqnarray}
Using (\ref{23}) and (\ref{ph}), we can derive
\begin{eqnarray}\label{projection}
\|P_hv\|_1 \lesssim \|v\|_1, \quad\forall v \in H_0^1(\Omega),
\end{eqnarray}
and
\begin{eqnarray}\label{Aubin-Nitsche}
\|v-P_hv\|_0\lesssim h\|v-P_hv\|_1\lesssim h\|v\|_1, \quad\forall v \in H_0^1(\Omega).
\end{eqnarray}
Combining (\ref{coercive of g1}), (\ref{ph}), (\ref{projection}) and (\ref{Aubin-Nitsche}) leads to
\begin{eqnarray}\label{isomorphism}
 \|(\gamma,w)\|_X &\lesssim&  \mathop {\sup} \limits_{(\mu,v)\in X}\dfrac{\langle\cG'(\lambda,u)(\gamma,w),(\mu,v)\rangle}{\|(\mu,v)\|_X}\nonumber\\
 &=&\mathop {\sup} \limits_{(\mu,v)\in X}\dfrac{\langle\cG'(\lambda,u)(\gamma,w),(0,v-P_hv)\rangle+
 \langle\cG'(\lambda,u)(\gamma,w),(\mu,P_hv)\rangle
 }{\|(\mu,v)\|_X}\nonumber\\
 &=&\mathop {\sup} \limits_{(\mu,v)\in X}\dfrac{a(\lambda, u;w,v-P_hv)
 -\gamma(u,v-P_hv)+\langle\cG'(\lambda,u)(\gamma,w),(\mu,P_hv)\rangle}{\|(\mu,v)\|_X}\nonumber\\
 &=&\mathop {\sup} \limits_{(\mu,v)\in X}\dfrac{-\gamma(u,v-P_hv)+\langle\cG'(\lambda,u)(\gamma,w),(\mu,P_hv)\rangle
 }{\|(\mu,v)\|_X}\nonumber\\
 &\lesssim&\mathop {\sup} \limits_{(\mu,v)\in X}\dfrac{\gamma \|u\|_0\|v-P_hv\|_0+\langle\cG'(\lambda,u)(\gamma,w),(\mu,P_hv)\rangle}{\|(\mu,v)\|_X}\nonumber\\
 &\lesssim&\mathop {\sup} \limits_{(\mu,v)\in X}\dfrac{\gamma h\|v\|_1+\langle\cG'(\lambda,u)(\gamma,w),(\mu,P_hv)\rangle}{\|(\mu,v)\|_X}\nonumber\\
 &\lesssim&   h\|(\gamma,w)\|_X+ \mathop {\sup} \limits_{(\mu,v)\in X}\dfrac{\langle\cG'(\lambda,u)(\gamma,w),(\mu,P_hv)\rangle}{\|(\mu,v)\|_X}.
 \end{eqnarray}
Thus, when the mesh size $h$ is small enough, we can obtain
\begin{eqnarray}\label{seco}
 \|(\gamma,w)\|_X &\lesssim&\mathop {\sup} \limits_{(\mu,v)\in X}\dfrac{\langle\cG'(\lambda,u)(\gamma,w),(\mu,P_hv)\rangle}{\|(\mu,v)\|_X}\nonumber\\
 &\lesssim&\mathop {\sup} \limits_{(\mu,v)\in X}\dfrac{\langle\cG'(\lambda,u)(\gamma,w),(\mu,P_hv)\rangle}{\|(\mu,P_hv)\|_X}\nonumber\\
 &\lesssim&\mathop {\sup} \limits_{(\mu,v)\in X_h}\dfrac{\langle\cG'(\lambda,u)(\gamma,w),(\mu,v)\rangle}{\|(\mu,v)\|_X},
 \end{eqnarray}
which is just the second formula (\ref{coercive of g2}).

For the last formula (\ref{coercive of g3}), we first prove a Lipschitz continuity for $\cG'$ based on the assumptions presented in Section 2,
H\"{o}lder inequality and the imbedding theorem (see \cite{Adams}):
\begin{eqnarray*}
&&\langle\cG'(\lambda,u)(\gamma,w),(\mu,v)\rangle-\langle\cG'(\widehat{\lambda},\widehat{u})(\gamma,w),(\mu,v)\rangle \\
&=&\big((f(u^2)-f(\widehat u^2))w,v\big)+((\widehat\lambda-\lambda)w,v)\\
&&+2\big((f'(u^2)u^2-f'(\widehat u^2)\widehat u^2)w,v\big) -\gamma(\widehat u-u,v)-\mu(\widehat u-u,w)\\
&=&2\big(f'(\xi^2)\xi(u-\widehat u)w,v\big)+((\widehat\lambda-\lambda)w,v)\\
&&+4\big((f''(\xi^2)\xi^3+f'(\xi^2)\xi)(u-\widehat u)w,v\big) -\gamma(\widehat u-u,v)-\mu(\widehat u-u,w)\\
&\lesssim&\big((1+\xi^{2q})\xi(u-\widehat u)w,v\big)+((\widehat\lambda-\lambda)w,v)\\
&&+\big((1+\xi^{2q})\xi(u-\widehat u)w,v\big) -\gamma(\widehat u-u,v)-\mu(\widehat u-u,w)\\
&\lesssim&\|\xi\|_{0,4}\|u-\widehat u\|_{0,4}\|w\|_{0,4}\|v\|_{0,4}+\|\xi^{2q}\|_{0,3/q}\|\xi\|_{0,6}\|u-\widehat u\|_{0,6}\|w\|_{0,6}\|v\|_{0,6/(3-2q)}\\
&&+|\widehat\lambda-\lambda|\|w\|_0\|v\|_0 +|\gamma|\|\widehat u-u\|_0\|v\|_0+|\mu|\|\widehat u-u\|_0\|w\|_0\\
&\lesssim&\|\xi\|_1\|u-\widehat u\|_1\|w\|_1\|v\|_1+\|\xi\|_{1}^{2q}\|\xi\|_1\|u-\widehat u\|_1\|w\|_1\|v\|_1\\
&&+|\widehat\lambda-\lambda|\|w\|_0\|v\|_0 +|\gamma|\|\widehat u-u\|_0\|v\|_0+|\mu|\|\widehat u-u\|_0\|w\|_0,
\end{eqnarray*}
where $\xi=\tau u+(1-\tau)\widehat u$ with $\tau\in(0,1)$.

Thus, we can derive
\begin{eqnarray}\label{lip}
&&\langle\cG'(\lambda,u)(\gamma,w),(\mu,v)\rangle-\langle\cG'(\widehat{\lambda},\widehat{u})(\gamma,w),(\mu,v)\rangle \nonumber\\
&\lesssim&\|(\lambda-\widehat{\lambda},u-\widehat{u})\|_X\|(\gamma,w)\|_X\|(\mu,v)\|_X.
\end{eqnarray}

Let $\varepsilon$ be a sufficiently small constant and $\|(\lambda-\widehat{\lambda},u-\widehat{u})\|_X \leq \varepsilon$.
Based on (\ref{coercive of g2}) and (\ref{lip}), for any $(\gamma,w)\in X_h$, we have
\begin{eqnarray*}
 \|(\gamma,w)\|_X &\lesssim&  \mathop {\sup} \limits_{(\mu,v)\in X_h}\dfrac{\varepsilon \|(\gamma,w)\|_X\|(\mu,v)\|_X+\langle\cG'(\widehat{\lambda},\widehat{u})(\gamma,w),(\mu,v)\rangle }{\|(\mu,v)\|_X}\\
 &\lesssim& \varepsilon \|(\gamma,w)\|_X +\mathop {\sup} \limits_{(\mu,v)\in X_h}\dfrac{\langle\cG'(\widehat{\lambda},\widehat{u})(\gamma,w),(\mu,v)\rangle}{\|(\mu,v)\|_X}.
 \end{eqnarray*}
Then $(\ref{coercive of g3})$ can be derived using the same way as the proof process in (\ref{seco}).
%there exists a constant $\widehat{\theta}$, such that
%\begin{eqnarray*}
% \|(\gamma,w)\|_X &\leq&  \widehat{\theta}\mathop {\sup} \limits_{(\sigma,\eta)\in X_h}\dfrac{\langle\cG'(\widehat{\lambda},\widehat{u})(\gamma,w),(\sigma,\eta)\rangle}{\|(\sigma,\eta)\|_X}.
%% &\lesssim&\mathop {\sup} \limits_{(\sigma,\eta)\in X_h}\dfrac{\langle\cG'(\widehat{\lambda},\widehat{u})(\gamma,w),(\sigma,\eta)\rangle }{\|(\sigma,\eta)\|_X}.
% \end{eqnarray*}
This completes the proof.
\end{proof}

Next, in Theorem \ref{residual estiamte}, we provide the {\color{blue}estimate} for the residual of the Newton iteration scheme (\ref{Newton Iteration}).
\begin{theorem}\label{residual estiamte}
For any $(\lambda^{(0)},u^{(0)})\in X, (\lambda^{(1)},u^{(1)})\in X$, we can derive the following expansion
\begin{eqnarray}\label{error of Newton Iteration}
\langle \cG(\lambda^{(1)},u^{(1)}),(\sigma,\eta)\rangle&=&\langle\cG(\lambda^{(0)},u^{(0)}),(\sigma,\eta)\rangle
+\langle \cG'(\lambda^{(0)},u^{(0)})(\lambda^{(1)}-\lambda^{(0)},u^{(1)}-u^{(0)}),(\sigma,\eta)\rangle \nonumber\\
&&+~R\big((\lambda^{(0)},u^{(0)}),(\lambda^{(1)},u^{(1)}),(\sigma,\eta)\big),\quad\quad\quad\quad\quad\quad\forall (\sigma,\eta)\in X,
\end{eqnarray}
where the residual $R\big((\lambda^{(0)},u^{(0)}),(\lambda^{(1)},u^{(1)}),(\sigma,\eta)\big)$ satisfies the following estimate
\begin{eqnarray*}
\big| R\big((\lambda^{(0)},u^{(0)}),(\lambda^{(1)},u^{(1)}),(\sigma,\eta)\big)\big| \lesssim \|(\lambda^{(0)}-\lambda^{(1)},u^{(0)}-u^{(1)})\|_X^2\|(\sigma,\eta)\|_X.
\end{eqnarray*}
\end{theorem}
\begin{proof}
We prove this theorem through the Taylor expansion. Let us define
\begin{eqnarray}
\eta(t)=\big\langle \cG\big((\lambda^{(0)},u^{(0)})+t\{(\lambda^{(1)},u^{(1)})-(\lambda^{(0)},u^{(0)})\}\big),(\sigma,\eta)\big\rangle, \ \ \ t \in [0,1].
\end{eqnarray}
Obviously,
\begin{eqnarray}
\eta(0)=\langle \cG(\lambda^{(0)},u^{(0)}),(\sigma,\eta)\rangle \ \ \ \text{and} \ \ \ \eta(1)=\langle \cG(\lambda^{(1)},u^{(1)}),(\sigma,\eta)\rangle.
\end{eqnarray}

The first order derivative and the second order derivative of $\eta(t)$ with respect to $t$ are listed as follows:
\begin{eqnarray}\label{firstderiv}
\eta'(t)&=&\big\langle \cG'((\lambda^{(0)},u^{(0)})+t\{(\lambda^{(1)},u^{(1)})-(\lambda^{(0)},u^{(0)})\})((\lambda^{(1)},u^{(1)})-(\lambda^{(0)},u^{(0)})),(\sigma,\eta)\big\rangle\nonumber\\
&=& \big\langle \cF_u'\big((\lambda^{(0)},u^{(0)})+t\{(\lambda^{(1)},u^{(1)})-(\lambda^{(0)},u^{(0)})\}\big)(u^{(1)}-u^{(0)}),\eta\big\rangle \nonumber\\
&&-(\lambda^{(1)}-\lambda^{(0)})(u^{(0)}+t(u^{(1)}-u^{(0)}),\eta)-\sigma(u^{(0)}+t(u^{(1)}-u^{(0)}),u^{(1)}-u^{(0)})\nonumber\\
&=& (\mathcal A\nabla(u^{(1)}-u^{(0)}),\nabla\eta) +\big(  (u^{(0)}+f((u^{(0)}+t(u^{(1)}-u^{(0)}))^2))(u^{(1)}-u^{(0)}),\eta  \big) \nonumber\\
&& - \big((\lambda^{(0)}+t(\lambda^{(1)}-\lambda^{(0)}))(u^{(1)}-u^{(0)}),\eta  \big) \nonumber\\
&&+2\big(f'((u^{(0)}+t(u^{(1)}-u^{(0)}))^2)(u^{(0)}+t(u^{(1)}-u^{(0)}))^2(u^{(1)}-u^{(0)}),\eta\big)\nonumber \\
&&-(\lambda^{(1)}-\lambda^{(0)})(u^{(0)}+t(u^{(1)}-u^{(0)}),\eta)-\sigma(u^{(0)}+t(u^{(1)}-u^{(0)}),u^{(1)}-u^{(0)}),
\end{eqnarray}
and
\begin{eqnarray}\label{secondderiv}
\eta''(t)&=&-2(\lambda^{(1)}-\lambda^{(0)})(u^{(1)}-u^{(0)},\eta)-\sigma (u^{(1)}-u^{(0)},u^{(1)}-u^{(0)}) \nonumber\\
&&+6\big(f'\big((u^{(0)}+t(u^{(1)}-u^{(0)}))^2\big)\big(u^{(0)}+t(u^{(1)}-u^{(0)})\big)(u^{(1)}-u^{(0)})^2,\eta\big)\nonumber\\
&&+4\big\langle f''\big((u^{(0)}+t(u^{(1)}-u^{(0)}))^2\big)\big(u^{(0)}+t(u^{(1)}-u^{(0)})\big)^3(u^{(1)}-u^{(0)})^2,\eta\big\rangle.
\end{eqnarray}

For the nonlinear terms involved in (\ref{secondderiv}), we have the following estimates
\begin{eqnarray}\label{f1}
&&\big|\big(f'\big((u^{(0)}+t(u^{(1)}-u^{(0)}))^2\big)\big(u^{(0)}+t(u^{(1)}-u^{(0)})\big)(u^{(1)}-u^{(0)})^2,\eta\big)\big|\nonumber\\
&\lesssim& \big|\big(\big(1+(u^{(0)}+t(u^{(1)}-u^{(0)}))^{2q}\big)\big(u^{(0)}+t(u^{(1)}-u^{(0)})\big)(u^{(1)}-u^{(0)})^2,\eta\big)\big|\nonumber\\
&\lesssim& \big|\big((u^{(0)}+t(u^{(1)}-u^{(0)}))(u^{(1)}-u^{(0)})^2,\eta\big)\nonumber\\
&& + \big((u^{(0)}+t(u^{(1)}-u^{(0)}))^{2q}(u^{(0)}+t(u^{(1)}-u^{(0)}))(u^{(1)}-u^{(0)})^2,\eta\big)\big|\nonumber\\
&\lesssim& \|u^{(0)}+t(u^{(1)}-u^{(0)})\|_{0,4}\|u^{(1)}-u^{(0)}\|_{0,4}^2\|\eta\|_{0,4}\nonumber\\
&&+\|(u^{(0)}+t(u^{(1)}-u^{(0)}))^{2q}\|_{0,3/q}\|u^{(0)}+t(u^{(1)}-u^{(0)})\|_{0,6}\|u^{(1)}-u^{(0)}\|_{0,6}^2\|\eta\|_{0,6/(3-2q)} \nonumber\\
&\lesssim& \|u^{(0)}+t(u^{(1)}-u^{(0)})\|_{0,4}\|u^{(1)}-u^{(0)}\|_{0,4}^2\|\eta\|_{0,4}\nonumber\\
&&+\|u^{(0)}+t(u^{(1)}-u^{(0)})\|_{0,6}^{2q}\|u^{(0)}+t(u^{(1)}-u^{(0)})\|_{0,6}\|u^{(1)}-u^{(0)}\|_{0,6}^2\|\eta\|_{0,6/(3-2q)} \nonumber\\
&\lesssim& \|u^{(0)}+t(u^{(1)}-u^{(0)})\|_1\|u^{(1)}-u^{(0)}\|_1^2\|\eta\|_1\nonumber\\
&&+\|u^{(0)}+t(u^{(1)}-u^{(0)})\|_1^{2q}\|u^{(0)}+t(u^{(1)}-u^{(0)})\|_1\|u^{(1)}-u^{(0)}\|_1^2\|\eta\|_1\nonumber\\
&\lesssim& \|u^{(1)}-u^{(0)}\|_1^2\|\eta\|_1,
\end{eqnarray}
and
\begin{eqnarray}\label{f2}
&&\big|\big\langle f''\big((u^{(0)}+t(u^{(1)}-u^{(0)}))^2\big)\big(u^{(0)}+t(u^{(1)}-u^{(0)})\big)^3(u^{(1)}-u^{(0)})^2,\eta\big\rangle\big|\nonumber\\
&\lesssim& \big|\big((1+(u^{(0)}+t(u^{(1)}-u^{(0)}))^{2q}\big)\big(u^{(0)}+t(u^{(1)}-u^{(0)})\big)(u^{(1)}-u^{(0)})^2,\eta)\big|\nonumber\\
&\lesssim& \|u^{(1)}-u^{(0)}\|_1^2\|\eta\|_1,
\end{eqnarray}
where the H\"{o}lder inequality and the imbedding theorem are used.

Based on (\ref{secondderiv}), (\ref{f1}) and (\ref{f2}), we can derive
\begin{eqnarray}\label{etest}
|\eta''(t)| &\lesssim& |\lambda^{(1)}-\lambda^{(0)}|\|u^{(1)}-u^{(0)}\|_0\|\eta\|_0 + |\sigma| \|u^{(1)}-u^{(0)}\|_0^2 +\|u^{(1)}-u^{(0)}\|_1^2\|\eta\|_1 \nonumber\\
&\lesssim& \|(\lambda^{(0)}-\lambda^{(1)},u^{(0)}-u^{(1)})\|_X^2\|(\sigma,\eta)\|_X.
\end{eqnarray}

Thus, using (\ref{etest}) and the following  Taylor expansion:
$$\eta(1)=\eta(0)+\eta'(0)+\int_{0}^{1}\eta''(t)(1-t)dt,$$
we can derive
\begin{eqnarray*}
\langle \cG(\lambda^{(1)},u^{(1)}),(\sigma,\eta)\rangle&=&\langle \cG(\lambda^{(0)},u^{(0)}),(\sigma,\eta)\rangle+\langle \cG'(\lambda^{(0)},u^{(0)})((\lambda^{(1)},u^{(1)})-(\lambda^{(0)},u^{(0)})),(\sigma,\eta)\rangle \\
&&+R((\lambda^{(0)},u^{(0)}),(\lambda^{(1)},u^{(1)}),(\sigma,\eta))
\end{eqnarray*}
with the residual
\begin{eqnarray*}
| R((\lambda^{(0)},u^{(0)}),(\lambda^{(1)},u^{(1)}),(\sigma,\eta))| \lesssim \|(\lambda^{(0)}-\lambda^{(1)},u^{(0)}-u^{(1)})\|_X^2\|(\sigma,\eta)\|_X.
\end{eqnarray*}
Then the proof is completed.
\end{proof}

\section{Multigrid method for nonlinear eigenvalue problems based on Newton iteration}
In this section, a novel multigrid method is proposed for solving the nonlinear eigenvalue problems based on the Newton iteration.
To design the multigrid method, let us construct a nested {\color{blue} sequence of finite element spaces} based on a nested multigrid sequence in the first step.
The sequence of meshes is denoted by $\{\mathcal T_{h_k}\}_{k=1}^n$. For any two consecutive meshes $\mathcal T_{h_k}$ and $\mathcal T_{h_{k-1}} (k\geq 2)$,
$\mathcal T_{h_k}$ is produced through a one-time uniform refinement form $\mathcal T_{h_{k-1}}$. The corresponding mesh sizes satisfy the following conditions:
\begin{eqnarray}
 h_k\approx \dfrac{1}{\beta}h_{k-1},\quad k=2,\cdots,n,
\end{eqnarray}
where the refinement index $\beta>1$. Meanwhile, the following approximate relationship holds true:
\begin{eqnarray}
\delta_{h_k}(u)\approx \dfrac{1}{\beta}\delta_{h_{k-1}}(u),\quad k=2,\cdots,n.
\end{eqnarray}
Based on the multigrid sequence $\{\mathcal T_{h_k}\}_{k=1}^n$, we denote the corresponding {\color{blue} sequence of finite element spaces} by:
\begin{eqnarray}
V_{h_{1}}\subset V_{h_{2}}\subset \cdots \subset V_{h_n}\subset H_0^1(\Omega),
\end{eqnarray}
and denote the {\color{blue}sequence of tensor product finite element spaces} by:
\begin{eqnarray}
X_{h_{1}}\subset X_{h_{2}}\subset \cdots \subset X_{h_n}\subset X.
\end{eqnarray}

%In this paper, we denote by $(\lambda_{h_k},u_{h_k})$ the standard finite element solution.

\subsection{One step of the multigrid method}
To describe the multigrid method more clearly, we first introduce how to perform one step
of the Newton iteration.

For two consecutive finite element spaces $X_{h_k}$ and $X_{h_{k+1}}$, assume that we have obtained an approximate solution $(\lambda^{h_{k}},u^{h_{k}})\in X_{h_k}$,
Algorithm \ref{onestep} shows the procedure for obtaining a new approximate solution $(\lambda^{h_{k+1}},u^{h_{k+1}})\in X_{h_{k+1}}$.

\begin{algorithm}\label{onestep}
One step of the multigrid method
\begin{enumerate}
\item Assume that we have obtained an approximate solution $(\lambda^{h_{k}},u^{h_{k}})\in X_{h_{k}}$.
\item Solve the linearized equation derived from Newton iteration in the finite element space $X_{h_{k+1}}$:
Find $(\lambda^{h_{k+1}},u^{h_{k+1}})\in X_{h_{k+1}}$ such that for any $(\mu,v_{h_{k+1}})\in X_{h_{k+1}}$, there holds
\begin{eqnarray}\label{dgd}
&&\langle \cG'(\lambda^{h_k},u^{h_k})(\lambda^{h_{k+1}},u^{h_{k+1}}),(\mu,v_{h_{k+1}})\rangle \nonumber\\
&=& -\langle \cG(\lambda^{h_{k}},u^{h_k}),(\mu,v_{h_{k+1}})\rangle
+\langle \cG'(\lambda^{h_k},u^{h_k})(\lambda^{h_k},u^{h_k}),(\mu,v_{h_{k+1}})\rangle.
\end{eqnarray}
The equation (\ref{dgd}) can be solved by the mixed finite element method in the following form: Find $(\lambda^{h_{k+1}},u^{h_{k+1}})\in \mathbb R\times  V_{h_{k+1}}$
such that for any $(\mu,v_{h_{k+1}})\in\mathbb R\times  V_{h_{k+1}}$, there holds
\begin{eqnarray*}
\left\{
\begin{array}{lll}
a(\lambda^{h_k}, u^{h_k};u^{h_{k+1}},v_{h_{k+1}})+b(u^{h_k};v_{h_{k+1}},\lambda^{h_{k+1}})
&=&\langle \cF_u'(\lambda^{h_k},u^{h_k})u^{h_k},v_{h_{k+1}}\rangle  \nonumber\\
&&-\langle \cF(\lambda^{h_k}, u^{h_k}),v_{h_{k+1}}\rangle-\lambda^{h_k}(u^{h_k},v_{h_{k+1}}),\\
b(u^{h_k};u^{h_{k+1}},\mu)&=&-\mu/2-\mu(u^{h_k},u^{h_k})/2.
\end{array}
\right.
\end{eqnarray*}
%\item Solve equation $(\ref{one correction bvp})$ to obtain an eigenpair approximation $(\lambda^{h_{k+1}},u^{h_{k+1}})$
%    satisfying $\|(\lambda^{h_{k+1}}-\widehat{\lambda}^{h_{k+1}},u^{h_{k+1}}-\widehat{u}^{h_{k+1}})\|_X\lesssim \eta_a(V_{h_{k+1}})\delta_{h_{k+1}}(u) $.
\end{enumerate}
  In order to simplify the notation, we use the following symbol to denote the above two solving steps:
  \begin{eqnarray*}
   (\lambda^{h_{k+1}},u^{h_{k+1}})=Newton_{-}Iteration(\lambda^{h_k}, u^{h_k},X_{h_{k+1}}).
  \end{eqnarray*}
 \end{algorithm}

Next, we can prove that the new approximate solution $(\lambda^{h_{k+1}},u^{h_{k+1}})\in X_{h_{k+1}}$
derived by Algorithm  \ref{onestep} has a better accuracy than $(\lambda^{h_{k}},u^{h_{k}})\in X_{h_k}$.
The detailed {\color{blue}conclusion is} presented in Theorem \ref{One Correction Step estimate}.

\begin{theorem}\label{One Correction Step estimate}
After implementing Algorithm $\ref{onestep}$, the new approximate solution $(\lambda^{h_{k+1}},u^{h_{k+1}})$ has the following error estimate
 \begin{eqnarray*}
 &&\|(\lambda-\lambda^{h_{k+1}},u-u^{h_{k+1}})\|_X \lesssim \delta_{h_{k+1}}(u)+\|(\lambda-\lambda^{h_k},u-u^{h_k})\|_X^2.
 \end{eqnarray*}
\end{theorem}
\begin{proof}
We use $(\lambda_{h_{k+1}},u_{h_{k+1}})$ to denote the standard finite element solution of (\ref{GPEfem}), that is $(\lambda_{h_{k+1}},u_{h_{k+1}})$ satisfies
\begin{eqnarray*}
&&(\cA\nabla u_{h_{k+1}},\nabla v_{h_{k+1}})+(Vu_{h_{k+1}},v_{h_{k+1}})+(f(u_{h_{k+1}}^2)u_{h_{k+1}},v_{h_{k+1}})\\
&=&\lambda_{h_{k+1}}(u_{h_{k+1}},v_{h_{k+1}}), \quad \forall v_{h_{k+1}} \in V_{h_{k+1}},
\end{eqnarray*}
or equivalently
\begin{eqnarray}\label{cdf}
\langle\cG(\lambda_{h_{k+1}},u_{h_{k+1}}),(\mu,v_{h_{k+1}})\rangle=0, \quad \forall (\mu,v_{h_{k+1}}) \in X_{h_{k+1}}.
\end{eqnarray}
%Through the second step of Algorithm 4.1,
%    \begin{eqnarray*}
%    &&\langle \cG'((\lambda_{h_{k+1}},v_{h_k}))(\lambda^{h_{k+1}},u^{h_{k+1}}),(\mu,v_{h_{k+1}})\rangle\rangle = -\langle \cG(\lambda^{h_k},u^{h_k}),(\mu,v_{h_{k+1}})\rangle \\
%    && \ \ \ \ \ \ \ \ \ \ \ \ \ \ \ \ \ \ \  \ \ \ \ \  \ \  \ \ \ \ \ \ \ \ \ \ \ \ \ \ \ \ \ \ \ \ \ \ \ \ \ \ +\langle \cG'(\lambda^{h_k},u^{h_k})(\lambda^{h_k},u^{h_k}),(\mu,v_{h_{k+1}})\rangle
%    \end{eqnarray*}

Combining (\ref{dgd}) and (\ref{cdf}) leads to
\begin{eqnarray}\label{midth}
&&\big\langle \cG'(\lambda^{h_k},u^{h_k})(\lambda_{h_{k+1}}-\lambda^{h_{k+1}},u_{h_{k+1}}-u^{h_{k+1}}),(\mu,v_{h_{k+1}})\big\rangle \nonumber\\
&=& \big\langle \cG(\lambda^{h_k},u^{h_k}),(\mu,v_{h_{k+1}})\big\rangle+ \big\langle \cG'(\lambda^{h_k},u^{h_k})(\lambda_{h_{k+1}}-\lambda^{h_k},u_{h_{k+1}}-u^{h_k}),(\mu,v_{h_{k+1}})\big\rangle \nonumber\\
&=& -\big\langle\cG(\lambda_{h_{k+1}},u_{h_{k+1}}),(\mu,v_{h_{k+1}})\big\rangle+
\big\langle \cG(\lambda^{h_k},u^{h_k}),(\mu,v_{h_{k+1}})\big\rangle\nonumber\\
&&+ \big\langle \cG'(\lambda^{h_k},u^{h_k})(\lambda_{h_{k+1}}-\lambda^{h_k},u_{h_{k+1}}-u^{h_k}),(\mu,v_{h_{k+1}})\big\rangle \nonumber\\
&=& -R((\lambda^{h_k},u^{h_k}), (\lambda_{h_{k+1}},u_{h_{k+1}}),(\mu,v_{h_{k+1}})).
\end{eqnarray}
Combining (\ref{coercive of g3}), Theorem \ref{residual estiamte} and (\ref{midth}), we can derive
\begin{eqnarray}\label{stad}
\|(\lambda_{h_{k+1}}-\lambda^{h_{k+1}},u_{h_{k+1}}-u^{h_{k+1}})\|_X\lesssim \|(\lambda_{h_{k+1}}-\lambda^{h_k},u_{h_{k+1}}-u^{h_k})\|_X^2.
\end{eqnarray}
Using the standard error estimates in Lemma \ref{stalem} and (\ref{stad}), we can obtain
\begin{eqnarray*}
\|(\lambda-\lambda^{h_{k+1}},u-u^{h_{k+1}})\|_X&\lesssim& \|(\lambda-\lambda_{h_{k+1}},u-u_{h_{k+1}})\|_X+\|(\lambda_{h_{k+1}}-\lambda^{h_{k+1}},u_{h_{k+1}}-u^{h_{k+1}})\|_X\nonumber\\
&\lesssim& \delta_{h_{k+1}}(u)+\|(\lambda-\lambda^{h_k},u-u^{h_k})\|_X^2.
\end{eqnarray*}
Then we complete the proof.
\end{proof}

\subsection{Multigrid method based on Newton iteration}
In this subsection, we introduce the multigrid method to solve the nonlinear eigenvalue problems.
We use the following nested {\color{blue} sequence of finite element spaces}:
\begin{eqnarray}
V_{h_{1}}\subset V_{h_{2}}\subset \cdots \subset V_{h_n}\subset H_0^1(\Omega),
\end{eqnarray}
and\begin{eqnarray}
X_{h_{1}}\subset X_{h_{2}}\subset \cdots \subset X_{h_n}\subset X.
\end{eqnarray}

The multigrid method is presented in Algorithm \ref{mgni}.  The new algorithm requires solving the nonlinear eigenvalue problem directly in the initial
space to obtain an initial value; then performing Algorithm $\ref{onestep}$ in the subsequent spaces. The approximate solution derived from the last space
is used as the initial iterative value in the current space.

\begin{algorithm}\label{mgni}
 Multigrid method based on Newton iteration
  \begin{enumerate}
    \item Solve the following nonlinear eigenvalue problem directly in the initial space $V_{h_1}$: Find $(\lambda^{h_1},u^{h_1})\in\mathbb R\times V_{h_1}$ such that $\|u^{h_1}\|_0=1$ and
     \begin{equation*}
      (\mathcal A\nabla u^{h_1},\nabla v_{h_1})+(Vu^{h_1},v_{h_1})+(f((u^{h_1})^2)u^{h_1},v_{h_1})=\lambda^{h_1}(u^{h_1},v_{h_1}), \quad\forall v_{h_1}\in V_{h_1}.
     \end{equation*}
    \item For $k=1,\cdots,n-1$, obtain a new approximate solution $(\lambda^{h_{k+1}},u^{h_{k+1}})\in X_{h_{k+1}}$ through
    \begin{eqnarray*}
   (\lambda^{h_{k+1}},u^{h_{k+1}})=Newton_{-}Iteration(\lambda^{h_k}, u^{h_k},X_{h_{k+1}}),
  \end{eqnarray*}
  End for.
  \item Finally, we obtain the approximate solution $(\lambda^{h_n},u^{h_n})\in X_{h_n}$.
  \end{enumerate}
 \end{algorithm}

The error estimate for the final approximate solution $(\lambda^{h_n},u^{h_n})\in X_{h_n}$ derived by Algorithm \ref{mgni} {\color{blue}is} presented in Theorem \ref{mgthm}.

\begin{theorem}\label{mgthm}
After performing Algorithm \ref{mgni}, the final approximate solution $(\lambda^{h_n},u^{h_n})\in X_{h_n}$ satisfies the following error estimate
 \begin{eqnarray}\label{ffe}
 \|(\lambda-\lambda^{h_n},u-u^{h_n})\|_X \lesssim \delta_{h_n}(u).
 \end{eqnarray}
\end{theorem}

\begin{proof}
We prove this theorem by the mathematical induction method.
Since we solve the nonlinear eigenvalue problem directly in the initial space, using Lemma \ref{stalem} leads to
\begin{eqnarray*}
 \|(\lambda-{\lambda}^{h_1},u-{u}^{h_1})\|_X \lesssim  \delta_{h_1}(u).
\end{eqnarray*}

Assume that (\ref{ffe}) holds true for $(\lambda^{h_{n-1}},{u}^{h_{n-1}})\in X_{h_{n-1}}$.
According to Theorem \ref{One Correction Step estimate}, we can derive
\begin{eqnarray*}\label{coercive in Newtopn}
&&\|(\lambda-{\lambda}^{h_{n}},u-{u}^{h_{n}})\|_X\nonumber\\
&\lesssim& \delta_{h_{n}}(u)+\|(\lambda-\lambda^{h_{n-1}},u-u^{h_{n-1}})\|_X^2\notag\\
&\lesssim&  \delta_{h_{n}}(u)+\delta_{h_{n-1}}^2(u)\notag\\
&\lesssim&  \delta_{h_{n}}(u).
\end{eqnarray*}
So (\ref{ffe}) also holds true for $(\lambda^{h_{n}},{u}^{h_{n}})\in X_{h_{n}}$.
Then the proof is completed.
\end{proof}

\subsection{Work estimate of Algorithm \ref{mgni}}
In this subsection, we end Section 4 by estimating the computational work of Algorithm \ref{mgni} briefly.
Let us use $N_k$ to denote the dimensions of $V_{h_k}$. The dimensions of the {\color{blue} sequence of finite element spaces} satisfy
 $$N_k \approx \beta^{d(k-n)}N_n,\ \ \ k=1,2,\cdots,n.$$
Then the computational work of Algorithm \ref{mgni} is presented in Theorem \ref{mixthm}, which shows that
Algorithm \ref{mgni} has linear computational complexity.
\begin{theorem}\label{mixthm}
Assume that solving the nonlinear eigenvalue problem directly in the initial space $V_{h_1}$ needs work $\cO(M_{h_1})$,
and solving the linearized boundary value problem (\ref{dgd}) in $V_{h_k}$ needs work $\cO(N_k)$ for $k=2,3,\cdots,n$.
Then the total computational work of Algorithm \ref{mgni} is $\cO(N_n+M_{h_1})$. Furthermore, the linear computational complexity
$\cO(N_n)$ can be derived provided $M_{h_1}\leq N_n$.
\end{theorem}

\begin{proof}
Let us use $W_k$ to denote the computational work of $V_{h_k}$.
Then the computational work $W$ of Algorithm \ref{mgni} can be estimated by:
\begin{eqnarray*}
W &=& \sum_{k=1}^{n}W_k = \cO(M_{h_1}+\sum_{k=2}^{n}N_k) \\
%&=& \cO( \sum_{k=2}^{n}N_k+M_{h_1} ) \\
&=& \cO( \sum_{k=2}^{n}\beta^{d(k-n)}N_n+M_{h_1} ) \\
&=& \cO(N_n+M_{h_1}).
\end{eqnarray*}
Then we get the desired conclusion $\cO(N_n+M_{h_1})$, and when $M_{h_1}\leq N_n$,  the total computational work changes to $\cO(N_n)$.
\end{proof}

\begin{remark}
The linearized boundary value problem (\ref{dgd}) can be solved efficiently by the multigrid method with
linear computational complexity $\cO(N_k)$  (see e.g., \cite{Saddle-book,Shaidurov}).
Because the dimension of the initial space $V_{h_1}$ is small and independent of the final finite element space,
it is easy to get $M_{h_1}\leq N_n$.
Thus, the linear computational complexity for Algorithm \ref{mgni} can be obtained.
\end{remark}

\section{Multigrid method for nonlinear eigenvalue problems based on mixing scheme}
In this section, an improved multigrid method is designed.
The motivation is that when solving the nonlinear equation (\ref{nonfem}) by Newton iteration, we may encounter
the non-convergence issue for some complicated models. This issue is the same as that of
the classical self-consistent field iteration for the nonlinear eigenvalue problems.
In order to overcome this difficulty, the mixing theme (damping method) is introduced to solve
the nonlinear eigenvalue problems.
The most widely used mixing scheme is Anderson acceleration \cite{andersion}, which is used to improve the convergence rate for fixed-point iterations.
Although Anderson acceleration was widely used for solving various models, the first mathematical convergence
result for linear and nonlinear problems did not appear until 2015 in \cite{Toth}.
For nonlinear eigenvalue problems, there is still no strictly theoretical analysis.
In general, the mixing scheme is used in fixed-point iterations to accelerate the convergence rate.
In this section, we use the mixing scheme in Newton iteration to design an improved multigrid method for nonlinear eigenvalue problems.
Most importantly, we can prove that the norm of the residual decreases in each step of the Newton iteration.
This may provide some inspiration to prove the theoretical conclusions for the
more general mixing scheme for nonlinear eigenvalue problems.

Assume that we have obtained an approximate solution $(\lambda^{h_{k}},u^{h_{k}})\in X_{h_{k}}$, we introduce a
novel iteration step in Algorithm \ref{mixscheme} to obtain a better approximation $(\lambda^{h_{k+1}},u^{h_{k+1}})\in X_{h_{k+1}}$
on the basis of Newton iteration and the mixing scheme.

\begin{algorithm}\label{mixscheme}
 One step of the mixing scheme
  \begin{enumerate}
  \item Given a parameter $\theta_{k+1}\in (0,1)$ and an approximate solution $(\lambda^{h_{k}},u^{h_{k}})\in X_{h_{k}}$.
\item Solve the linearized equation derived from Newton iteration in the finite element space $X_{h_{k+1}}$:
Find $(\widehat \lambda^{h_{k+1}},\widehat u^{h_{k+1}})\in X_{h_{k+1}}$ such that for any $(\mu,v_{h_{k+1}})\in X_{h_{k+1}}$, there holds
\begin{eqnarray}\label{dgdmix}
&&\langle \cG'(\lambda^{h_k},u^{h_k})(\widehat \lambda^{h_{k+1}},\widehat u^{h_{k+1}}),(\mu,v_{h_{k+1}})\rangle \nonumber\\
&=& -\langle \cG(\lambda^{h_{k}},u^{h_k}),(\mu,v_{h_{k+1}})\rangle
+\langle \cG'(\lambda^{h_k},u^{h_k})(\lambda^{h_k},u^{h_k}),(\mu,v_{h_{k+1}})\rangle.
\end{eqnarray}
The equation (\ref{dgdmix}) can be solved by the mixed finite element method in the following form: Find $(\widehat \lambda^{h_{k+1}},\widehat u^{h_{k+1}})\in \mathbb R\times  V_{h_{k+1}}$
such that for any $(\mu,v_{h_{k+1}})\in\mathbb R\times  V_{h_{k+1}}$, there holds
\begin{eqnarray*}
\left\{
\begin{array}{lll}
a(\lambda^{h_k}, u^{h_k};\widehat u^{h_{k+1}},v_{h_{k+1}})+b(u^{h_k};v_{h_{k+1}},\widehat \lambda^{h_{k+1}})
&=&\langle \cF_u'(\lambda^{h_k},u^{h_k})u^{h_k},v_{h_{k+1}}\rangle  \nonumber\\
&&-\langle \cF(\lambda^{h_k}, u^{h_k}),v_{h_{k+1}}\rangle-\lambda^{h_k}(u^{h_k},v_{h_{k+1}}),\\
b(u^{h_k};\widehat u^{h_{k+1}},\mu)&=&-\mu/2-\mu(u^{h_k},u^{h_k})/2.
\end{array}
\right.
\end{eqnarray*}
\item Set  \begin{eqnarray*}
    (\lambda^{h_{k+1}},u^{h_{k+1}}) = (1-\theta_{k+1})(\lambda^{h_k},u^{h_k}) +\theta_{k+1}(\widehat\lambda^{h_{k+1}},\widehat u^{h_{k+1}}).
     \end{eqnarray*}
\end{enumerate}
 \end{algorithm}

Next, we can prove that the value of the residual $|\langle \cG(\lambda^{h_{k+1}},u^{h_{k+1}}),(\mu,v_{h_{k+1}})\rangle|$
monotonically decreases after performing Algorithm \ref{mixscheme}.

\begin{theorem}\label{monotonical}
Assume that we have obtained an approximate solution $(\lambda^{h_{k}},u^{h_{k}})$, then, there exists a sufficiently small $\theta_{k+1}$ such that the
new approximation $(\lambda^{h_{k+1}},u^{h_{k+1}})$ derived by Algorithm \ref{mixscheme} has the following {\color{blue}convergence}
\begin{eqnarray*}
|\langle \cG(\lambda^{h_{k+1}},u^{h_{k+1}}),(\mu,v_{h_{k+1}})\rangle| \leq (1-\frac{\theta_{k+1}}{2})|\langle \cG(\lambda^{h_{k}},u^{h_k}),(\mu,v_{h_{k+1}})\rangle|, \ \ \forall (\mu,v_{h_{k+1}})\in X_{h_{k+1}}.
 \end{eqnarray*}
\end{theorem}

\begin{proof}
From the following mixing scheme
\begin{eqnarray*}
(\lambda^{h_{k+1}},u^{h_{k+1}}) = (1-\theta_{k+1})(\lambda^{h_k},u^{h_k}) +\theta_{k+1}(\widehat\lambda^{h_{k+1}},\widehat u^{h_{k+1}}),
\end{eqnarray*}
used in Algorithm \ref{mixscheme}, we can derive
\begin{eqnarray}\label{scd}
(\lambda^{h_{k+1}}-\lambda^{h_k},u^{h_{k+1}}-u^{h_k}) = \theta_{k+1}(\widehat\lambda^{h_{k+1}}-\lambda^{h_k},\widehat u^{h_{k+1}}-u^{h_k}).
\end{eqnarray}

Combining (\ref{scd}) and the mean value theorem leads to
\begin{eqnarray}\label{dfd}
&&\langle \cG(\lambda^{h_{k+1}},u^{h_{k+1}}),(\mu,v_{h_{k+1}})\rangle -\langle \cG(\lambda^{h_{k}},u^{h_{k}}),(\mu,v_{h_{k+1}})\rangle \nonumber\\
&&-\langle \cG'(\lambda^{h_k},u^{h_k})(\lambda^{h_{k+1}}-\lambda^{h_k},u^{h_{k+1}}-u^{h_k}),(\mu,v_{h_{k+1}})\rangle \nonumber\\
&=& \langle\cG'(\widetilde\lambda^{h_k},\widetilde u^{h_k})(\lambda^{h_{k+1}}-\lambda^{h_k}, u^{h_{k+1}}-u^{h_k}),(\mu,v_{h_{k+1}})\rangle \nonumber\\
&&-\langle \cG'(\lambda^{h_k},u^{h_k})(\lambda^{h_{k+1}}-\lambda^{h_k}, u^{h_{k+1}}-u^{h_k}),(\mu,v_{h_{k+1}})\rangle \nonumber\\
%&=& \cG'(\alpha\lambda^{h_{k+1}}+(1-\alpha)\lambda^{h_k},\alpha u^{h_{k+1}}+(1-\alpha)u^{h_k})(\lambda^{h_{k+1}}-\lambda^{h_k},u^{h_{k+1}}-u^{h_k}),(\mu,v_{h_{k+1}}) \\
%&&-\theta\langle \cG'(\lambda^{h_k},u^{h_k})(\widehat\lambda^{h_{k+1}}-\lambda^{h_k},\widehat u^{h_{k+1}}-u^{h_k}),(\mu,v_{h_{k+1}})\rangle \\
&=& \theta_{k+1}\langle\cG'(\widetilde\lambda^{h_k},\widetilde u^{h_k})(\widehat\lambda^{h_{k+1}}-\lambda^{h_k},\widehat u^{h_{k+1}}-u^{h_k}),(\mu,v_{h_{k+1}})\rangle \nonumber\\
&&-\theta_{k+1}\langle \cG'(\lambda^{h_k},u^{h_k})(\widehat\lambda^{h_{k+1}}-\lambda^{h_k},\widehat u^{h_{k+1}}-u^{h_k}),(\mu,v_{h_{k+1}})\rangle,
\end{eqnarray}
where $(\widetilde\lambda^{h_k},\widetilde u^{h_k})=(\alpha\lambda^{h_{k+1}}+(1-\alpha)\lambda^{h_k},\alpha u^{h_{k+1}}+(1-\alpha)u^{h_k})$ with $\alpha\in (0,1)$.

From (\ref{scd}), we know that when $\theta_{k+1}\rightarrow 0$, there holds
\begin{eqnarray*}
(\lambda^{h_{k+1}},u^{h_{k+1}})\rightarrow(\lambda^{h_{k}},u^{h_{k}}),
\end{eqnarray*}
and thus
\begin{eqnarray*}
(\widetilde\lambda^{h_k},\widetilde u^{h_k})\rightarrow(\lambda^{h_{k}},u^{h_{k}}),
\end{eqnarray*}
which also indicates
\begin{eqnarray*}
&&\big\langle\cG'(\widetilde\lambda^{h_k},\widetilde u^{h_k})(\widehat\lambda^{h_{k+1}}-\lambda^{h_k},\widehat u^{h_{k+1}}-u^{h_k}),(\mu,v_{h_{k+1}})\big\rangle \\
&\rightarrow& \big\langle \cG'(\lambda^{h_k},u^{h_k})(\widehat\lambda^{h_{k+1}}-\lambda^{h_k},\widehat u^{h_{k+1}}-u^{h_k}),(\mu,v_{h_{k+1}})\big\rangle.
\end{eqnarray*}

Hence, we can choose a sufficiently small $\theta_{k+1}$, such that
\begin{eqnarray}\label{fvx}
&&\big|\langle\cG'(\widetilde\lambda^{h_k},\widetilde u^{h_k})(\widehat\lambda^{h_{k+1}}-\lambda^{h_k},\widehat u^{h_{k+1}}-u^{h_k}),(\mu,v_{h_{k+1}})\rangle\nonumber\\
&&- \langle \cG'(\lambda^{h_k},u^{h_k})(\widehat\lambda^{h_{k+1}}-\lambda^{h_k},\widehat u^{h_{k+1}}-u^{h_k}),(\mu,v_{h_{k+1}})\rangle\big|\nonumber\\
&\leq& \frac{1}{2}\big|\langle \cG'(\lambda^{h_k},u^{h_k})(\widehat\lambda^{h_{k+1}}-\lambda^{h_k},\widehat u^{h_{k+1}}-u^{h_k}),(\mu,v_{h_{k+1}})\rangle\big|\nonumber\\
&=& \frac{1}{2}\big|\langle \cG(\lambda^{h_k},u^{h_k}),(\mu,v_{h_{k+1}})\rangle\big|.
\end{eqnarray}

Based on (\ref{dfd}) and (\ref{fvx}), we can derive
\begin{eqnarray}\label{pio1}
&&\big|\langle\cG(\lambda^{h_{k+1}},u^{h_{k+1}}),(\mu,v_{h_{k+1}})\rangle -\langle \cG(\lambda^{h_{k}},u^{h_{k}}),(\mu,v_{h_{k+1}})\rangle \nonumber\\
&&-\langle \cG'(\lambda^{h_k},u^{h_k})(\lambda^{h_{k+1}}-\lambda^{h_k},u^{h_{k+1}}-u^{h_k}),(\mu,v_{h_{k+1}})\rangle\big| \nonumber\\
&\leq& \frac{\theta_{k+1}}{2}|\langle \cG(\lambda^{h_k},u^{h_k}),(\mu,v_{h_{k+1}})\rangle\big|.
\end{eqnarray}

Besides, using (\ref{dgdmix}) and (\ref{scd}), we can also derive
\begin{eqnarray}\label{pio2}
&&\big|\langle\cG(\lambda^{h_{k+1}},u^{h_{k+1}}),(\mu,v_{h_{k+1}})\rangle -\langle \cG(\lambda^{h_{k}},u^{h_{k}}),(\mu,v_{h_{k+1}})\rangle \nonumber\\
&&-\langle \cG'(\lambda^{h_k},u^{h_k})(\lambda^{h_{k+1}}-\lambda^{h_k},u^{h_{k+1}}-u^{h_k}),(\mu,v_{h_{k+1}})\rangle\big| \nonumber\\
&=&\big|\langle\cG(\lambda^{h_{k+1}},u^{h_{k+1}}),(\mu,v_{h_{k+1}})\rangle -\langle \cG(\lambda^{h_{k}},u^{h_{k}}),(\mu,v_{h_{k+1}})\rangle \nonumber\\
&&-\theta_{k+1}\langle \cG'(\lambda^{h_k},u^{h_k})(\widehat\lambda^{h_{k+1}}-\lambda^{h_k},\widehat u^{h_{k+1}}-u^{h_k}),(\mu,v_{h_{k+1}})\rangle\big| \nonumber\\
&=&\big|\langle\cG(\lambda^{h_{k+1}},u^{h_{k+1}}),(\mu,v_{h_{k+1}})\rangle -\langle \cG(\lambda^{h_{k}},u^{h_{k}}),(\mu,v_{h_{k+1}})\rangle \nonumber\\
&&+\theta_{k+1}\langle \cG(\lambda^{h_k},u^{h_k}),(\mu,v_{h_{k+1}})\rangle\big| \nonumber\\
&=&\big|\langle\cG(\lambda^{h_{k+1}},u^{h_{k+1}}),(\mu,v_{h_{k+1}})\rangle -(1-\theta_{k+1})\langle \cG(\lambda^{h_{k}},u^{h_{k}}),(\mu,v_{h_{k+1}})\rangle\big|.
\end{eqnarray}

Combining (\ref{pio1}) and (\ref{pio2}) leads to
\begin{eqnarray*}
&&\big|\langle\cG(\lambda^{h_{k+1}},u^{h_{k+1}}),(\mu,v_{h_{k+1}})\rangle -(1-\theta_{k+1})\langle \cG(\lambda^{h_{k}},u^{h_{k}}),(\mu,v_{h_{k+1}})\rangle\big| \\
&\leq& \frac{\theta_{k+1}}{2}|\langle \cG(\lambda^{h_k},u^{h_k}),(\mu,v_{h_{k+1}})\rangle|.
\end{eqnarray*}
That is
\begin{eqnarray*}
&&|\langle\cG(\lambda^{h_{k+1}},u^{h_{k+1}}),(\mu,v_{h_{k+1}})\rangle| \\
&\leq& (1-\frac{\theta_{k+1}}{2})|\langle \cG(\lambda^{h_k},u^{h_k}),(\mu,v_{h_{k+1}})\rangle|.
\end{eqnarray*}
Then we complete the proof.
\end{proof}

In order to give a practical algorithm, there are two remaining problems. The first one is to choose an appropriate parameter $\theta_{k+1}$.
The second one is to compute the residual after each iteration step.

For the first problem, we provide an adaptive strategy for choosing $\theta_{k+1}$ in Algorithm \ref{mix1}, which gradually reduces the value of $\theta_{k+1}$.
For the second problem, in order to compare the following two values $|\langle \cG(\lambda^{h_{k+1}},u^{h_{k+1}}),(\mu,v_{h_{k+1}})\rangle|$ and $|\langle \cG(\lambda^{h_{k}},u^{h_k}),(\mu,v_{h_{k+1}})\rangle|$
for any $(\mu,v_{h_{k+1}})\in X_{h_{k+1}}$,
we define the residual for the nonlinear eigenvalue problem (\ref{Nonlinear_Eigenvalue_Problem}) directly in the following way:
\begin{eqnarray*}
Resi(\lambda^{h_{k}},u^{h_k})= \|A_{u^{h_k}}u^{h_k}-\lambda^{h_{k}}u^{h_k}\|_1+\frac{1}{2}|1-\|u^{h_{k}}\|_0^2|.
\end{eqnarray*}
Then next, we use $Resi(\lambda^{h_{k}},u^{h_k})$ to judge the error reduction.

The modified mixing scheme and the corresponding multigrid method are presented in Algorithms \ref{mix1} and \ref{mgmix}, respectively.
\begin{algorithm}\label{mix1}
 One step of the modified mixing scheme
  \begin{enumerate}
    \item Assume that we have obtained an approximate solution $(\lambda^{h_{k}},u^{h_{k}})\in X_{h_{k}}$.
    \item Set $\theta_{k+1}=1$.
    \item Solve the linearized equation derived from Newton iteration in the finite element space $X_{h_{k+1}}$:
Find $(\widehat \lambda^{h_{k+1}},\widehat u^{h_{k+1}})\in X_{h_{k+1}}$ such that for any $(\mu,v_{h_{k+1}})\in X_{h_{k+1}}$, there holds
     \begin{eqnarray*}
     \langle \cG'(\lambda^{h_k},u^{h_k})(\widehat \lambda^{h_{k+1}},\widehat u^{h_{k+1}}),(\mu,v_{h_{k+1}})\rangle &=& -\langle \cG(\lambda^{h_{k}},u^{h_k}),(\mu,v_{h_{k+1}})\rangle \\
     && +\langle \cG'(\lambda^{h_k},u^{h_k})(\lambda^{h_k},u^{h_k}),(\mu,v_{h_{k+1}})\rangle.
     \end{eqnarray*}
The above equation can be solved by the mixed finite element method in the following form: Find $(\widehat \lambda^{h_{k+1}},\widehat u^{h_{k+1}})\in \mathbb R\times  V_{h_{k+1}}$
such that for any $(\mu,v_{h_{k+1}})\in\mathbb R\times  V_{h_{k+1}}$, there holds
\begin{eqnarray*}
\left\{
\begin{array}{lll}
a(\lambda^{h_k}, u^{h_k};\widehat u^{h_{k+1}},v_{h_{k+1}})+b(u^{h_k};v_{h_{k+1}},\widehat \lambda^{h_{k+1}})
&=&\langle \cF_u'(\lambda^{h_k},u^{h_k})u^{h_k},v_{h_{k+1}}\rangle \nonumber\\
&&-\langle \cF(\lambda^{h_k}, u^{h_k}),v_{h_{k+1}}\rangle-\lambda^{h_k}(u^{h_k},v_{h_{k+1}}),\\
b(u^{h_k};\widehat u^{h_{k+1}},\mu)&=&-\mu/2-\mu(u^{h_k},u^{h_k})/2.
\end{array}
\right.
\end{eqnarray*}
\item Set  \begin{eqnarray*}
    (\lambda^{h_{k+1}},u^{h_{k+1}}) = (1-\theta_{k+1})(\lambda^{h_k},u^{h_k}) +\theta_{k+1}(\widehat\lambda^{h_{k+1}},\widehat u^{h_{k+1}}).
     \end{eqnarray*}
\item If $Resi(\lambda^{h_{k+1}},u^{h_{k+1}}) \leq Resi(\lambda^{h_{k}},u^{h_k})$, stop. Else set $\theta_{k+1}=\theta_{k+1}/2$ and goto Step 4.
\end{enumerate}
In order to simplify the notation, we use the following symbol to denote the above solving steps:
  \begin{eqnarray*}
   (\lambda^{h_{k+1}},u^{h_{k+1}})=Mixing_{-}Iteration(\lambda^{h_k}, u^{h_k},\theta_{k+1},X_{h_{k+1}}).
  \end{eqnarray*}
 \end{algorithm}

\begin{algorithm}\label{mgmix}
 Multigrid method based on mixing scheme
  \begin{enumerate}
    \item Solve the following nonlinear eigenvalue problem directly in the initial space $V_{h_1}$: Find $(\lambda^{h_1},u^{h_1})\in\mathbb R\times V_{h_1}$ such that $\|u^{h_1}\|_0=1$ and
     \begin{equation*}
      (\mathcal A\nabla u^{h_1},\nabla v_{h_1})+(Vu^{h_1},v_{h_1})+(f((u^{h_1})^2)u^{h_1},v_{h_1})=\lambda^{h_1}(u^{h_1},v_{h_1}), \quad\forall v_{h_1}\in V_{h_1}.
     \end{equation*}
    \item For $k=1,\cdots,n-1$, obtain a new approximate solution $(\lambda^{h_{k+1}},u^{h_{k+1}})\in X_{h_{k+1}}$ through
    \begin{eqnarray*}
   (\lambda^{h_{k+1}},u^{h_{k+1}})=Mixing_{-}Iteration(\lambda^{h_k}, u^{h_k},\theta_{k+1},X_{h_{k+1}}),
  \end{eqnarray*}
  End for.
  \item Finally, we obtain the approximate solution $(\lambda^{h_n},u^{h_n})\in X_{h_n}$.
  \end{enumerate}
 \end{algorithm}

\section{Numerical results}
In this section, we propose some numerical experiments to support our theoretical conclusions and demonstrate the
efficiency of the presented multigrid methods.
For the linearized boundary value problems derived by Newton iteration, we used the V-cycle multigrid method to obtain the numerical solutions.
The multigrid method includes two pre-smoothing steps and two post-smoothing steps. The adopted smoother for the pre-smoothing and post-smoothing
steps is the distributive Gauss-Seidel (DGS) iteration \cite{Saddle-book}.

\subsection{Example 1}
In the first example, we use Algorithm \ref{mgni} to solve the Gross-Pitaevskii equation: Find $(\lambda, u)$
such that
\begin{equation}\label{nonlinear_pde}
\left\{
\begin{array}{rcl}
-\triangle u+Vu +\zeta|u|^2 u&=&\lambda u, \quad {\rm in} \  \Omega,\\
u&=&0, \ \  \quad {\rm on}\  \partial\Omega,\\
\int_{\Omega}u^2d\Omega&=&1,
\end{array}
\right.
\end{equation}
where $\Omega=[0,1]^3$, $V=x_1^2+2x_2^2+4x_3^2$ and $\zeta=1$.

%The sequence of finite element spaces are constructed by linear element on a series of meshes
%produced by regular refinement with $\beta=2$ (producing $\beta^3$ congruent subelements).
%Since the exact solution is not known, an adequate accurate approximation is choosen as the exact solution
%for our numerical test.

The quadratic finite element space was adopted in this example. The sequence of meshes was produced through a one-time uniform refinement.
Thus, the refinement index $\beta$ between two consecutive meshes equals $2$.
The initial mesh is depicted in Figure \ref{mesh1}.

\begin{figure}[ht!]
  \centering
  \includegraphics[width=8cm]{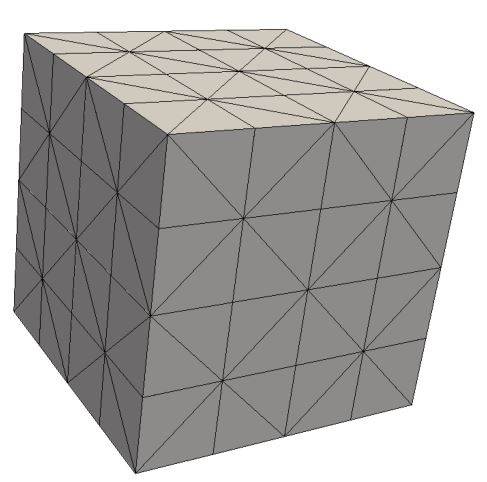}
  \caption{ The initial mesh for Example 1.}\label{mesh1}
  \label{mesh1}
\end{figure}

Because the exact solution of (\ref{nonlinear_pde}) is unknown, we select an adequately accurate approximation on a
sufficiently fine mesh as the exact solution ($\lambda=34.819449$).
The error estimates of the approximate solutions derived by Algorithm \ref{mgni} are presented in Figure \ref{ex1}.
The results show that Algorithm \ref{mgni} is able to derive the approximate solutions with the optimal error estimates.
In order to intuitively illustrate the efficiency of Algorithm \ref{mgni},
the computational time of Algorithm \ref{mgni} is presented in Table 1.
Besides, the computational time of  the direct finite element method (solve the nonlinear eigenvalue problem directly in the final finite element space) is also presented.
We can see that the linear computational complexity of Algorithm \ref{mgni} can be obtained with the refinement of mesh. In addition, Algorithm \ref{mgni} has
a great advantage over the direct finite element method.

\begin{figure}[ht!]
  \centering
  \includegraphics[width=6.7cm,height=5.6cm]{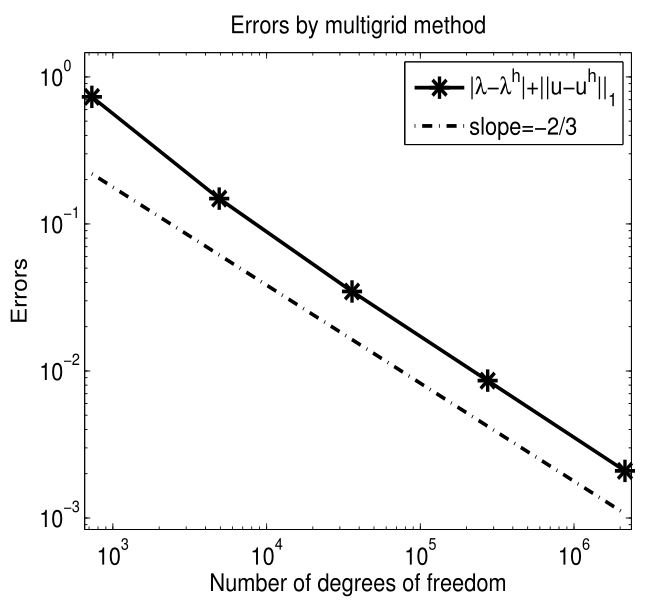}
  \caption{\small\texttt  The errors of approximate solutions derived by Algorithm \ref{mgni} for Example 1.}
           \label{ex1}
  %\caption{}%\label{time of bec1}
\end{figure}

\begin{table}[ht!]%表格浮动环境
\begin{center}
\caption{The CPU time of Algorithm \ref{mgni} and  the direct finite element method for Example 1. }
\begin{tabular}{|c|c|c|c|}\hline
Mesh level  & Number of degrees of freedom & Time of Algorithm \ref{mgni}& Time of direct FEM \\ \hline
%1 & 384           & 0.11        &   0.11       \\ \hline
1 & 729          & 0.1643    & 0.1643           \\ \hline
2 & 4913         & 0.2642    & 1.1434         \\ \hline
3 & 35937        & 1.1011    & 12.2292         \\ \hline
4 & 274625       & 8.3957    & 151.3532        \\ \hline
5 & 2146689      & 66.9435   & 3947.5318        \\ \hline
6 & 16974593     & 540.2690  & -      \\ \hline
\end{tabular}
\end{center}
\label{table1}
\end{table}

\subsection{Example 2}
 In the second example, we use Algorithm \ref{mgmix} to solve
the Gross-Pitaevskii equation (\ref{nonlinear_pde}) in $\Omega = [0,1]^3$, where $W=x_1^2+x_2^2+x_3^2+sin^2(2\pi x_1)+sin^2(2\pi x_2)+sin^2(2\pi x_3)$ and $\zeta=100$.
%The initial mesh used in this example is shown in Figure \ref{mesh2}.

In this example, due to the strong nonlinearity, Algorithm \ref{mgni} does not converge. Thus, we used Algorithm \ref{mgmix}
to solve the nonlinear eigenvalue problem and the convergent results are obtained.

The quadratic finite element space was adopted in this example. The sequence  of meshes was produced through a one-time uniform refinement.
Thus, the refinement index $\beta$ between two consecutive meshes equals $2$. The initial mesh is depicted in Figure \ref{mesh2}.

\begin{figure}[ht!]
  \centering
  \includegraphics[width=8cm]{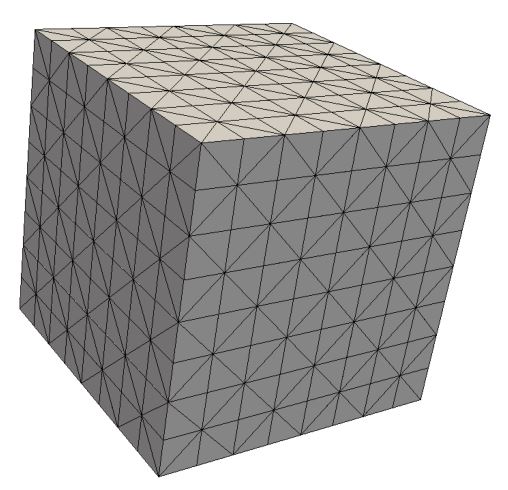}
  \caption{ The initial mesh for Example 2.}\label{mesh2}
  \label{mesh2}
\end{figure}

\begin{figure}[ht!]
  \centering
  \includegraphics[width=6.7cm,height=5.6cm]{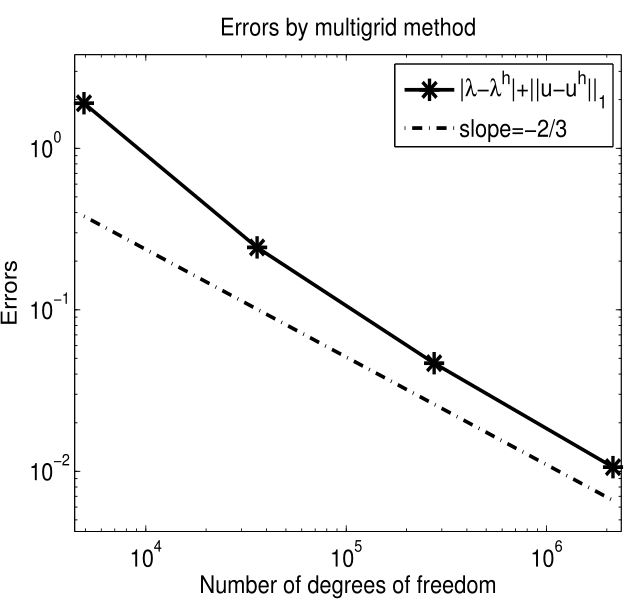}
  \caption{\small\texttt  The errors of approximate solutions derived by Algorithm \ref{mgmix} for Example 2.}
           \label{ex2}
  %\caption{}%\label{time of bec1}
\end{figure}

Because the exact solution of (\ref{nonlinear_pde}) is unknown, we select an adequately accurate approximation on a
sufficiently fine mesh as the exact solution ($\lambda=205.112532$).
The error estimates of the approximate solutions derived by Algorithm \ref{mgmix} are shown in Figure \ref{ex2}.
The results show that Algorithm \ref{mgmix} is able to derive the approximate solutions with the optimal error estimates.
In order to intuitively illustrate the efficiency of Algorithm \ref{mgmix},
the computational time of Algorithm \ref{mgmix} and the direct finite element method {\color{blue}are} presented in Table 2.
We can see that the linear computational complexity can be obtained with the refinement of mesh. In addition,
Algorithm \ref{mgmix} has a great advantage over the direct finite element method.

In Table 3, we also present the value of $\theta_k$ in each finite element space $X_{h_k}$. The residuals corresponding to different $\theta_k$
are also presented. Based on the adaptive strategy, we can see that the value of the residual monotonically decreases with the refinement of mesh, and the
results are consistent with Theorem \ref{monotonical}.

\begin{table}[ht!]%表格浮动环境
\begin{center}
\caption{The CPU time of Algorithm \ref{mgmix} and the direct finite element method for Example 2.}
\begin{tabular}{|c|c|c|c|}\hline
Mesh level  & Number of degrees of freedom & Time of Algorithm \ref{mgmix}& Time of direct FEM \\ \hline
1 & 4913         & 2.1944     & 2.1944              \\ \hline
2 & 35937        & 3.0003     & 24.5324             \\ \hline
3 & 274625       & 10.7252    & 289.6608              \\ \hline
4 & 2146689      & 80.1640    & 7560.2502             \\ \hline
5 & 16974593     & 676.3930   & -              \\ \hline
\end{tabular}
\end{center}
\label{t2}
\end{table}

\begin{table}[ht!]\label{t3}
\begin{center}
\caption{The value of $\theta_k$ in Algorithm \ref{mgmix} for Example 2. }
\begin{tabular}{|c|c|c|}\hline
Mesh level  & The value of $\theta_k$ & The value of $Resi$ \\ \hline
1         & 0.5 & 17.8814              \\ \hline
2         & 0.5 & 8.9953            \\ \hline
3       & 0.5 & 4.4768             \\ \hline
4       & 0.5 & 2.2358            \\ \hline
5      & 0.5 & 1.1355            \\ \hline
\end{tabular}
\end{center}
\end{table}
\section{Concluding remarks}
In this study, we designed a novel multigrid method to solve the nonlinear eigenvalue problem
on the basis of the Newton iteration. The novel scheme transforms the nonlinear eigenvalue problem into
a series of linearized boundary value problems defined in a {\color{blue}sequence of product finite element spaces}.
Because of avoiding solving large-scale nonlinear eigenvalue problems directly, the overall solving efficiency is significantly improved.
Besides, the optimal error estimate and linear
complexity can be derived simultaneously. In addition, an improved multigrid method coupled with the mixing scheme is also
introduced. The improved scheme can make the iteration scheme converge for more complicated models.
More importantly,  a convergence result is derived which is missing in the existing literature on the mixing scheme for nonlinear eigenvalue problems.

\section*{Declarations}
{\bf Conflict of interest}: The authors declare that they have no conflicts of interest/competing interests.\\
{\bf Code Availability}: The custom codes generated during the current study are available from the corresponding author on reasonable request.

\end{document}